\documentclass{article}

\usepackage{amsfonts}
\usepackage{amsthm}
\usepackage{amsmath}
\usepackage{geometry}
\usepackage{tikz}
\usepackage{tikz-cd}
\usepackage[all]{xy}
\usetikzlibrary{matrix, calc, arrows}
\usepackage[varg]{txfonts}
\usepackage{bm}
\usepackage{enumitem}
\usepackage{indentfirst}
\usepackage{comment}
\usepackage{ulem}
\usepackage{cases}
\usepackage{authblk}
\usepackage{hyperref}

\newtheorem{thm}{Theorem}[section]
\newtheorem{rem}[thm]{Remark}

\newtheorem{prop}[thm]{Proposition}

\newtheorem{cor}[thm]{Corollary}

\newtheorem{prop-defi}[thm]{Proposition and Definition}

\theoremstyle{definition}
\newtheorem{defn}[thm]{Definition}

\def\Zt{{\mathbb{Z}}}
\def\NN{{\mathbb{N}}}

\DeclareMathOperator{\im}{\mathrm{im}}

\DeclareMathOperator{\obj}{\mathrm{obj}}
\DeclareMathOperator{\Hom}{\mathrm{Hom}}

\newcommand{\sets}{\mathbf{Sets}}
\newcommand{\hgrh}{\mathbf{hGraph}}
\newcommand{\grh}{\mathbf{Graph}}
\newcommand{\htop}{\mathbf{hTop}}
\newcommand{\id}{\mathrm{id}}

\title{Mapping fiber graphs and loop graphs in naive discrete homotopy theory}
\author[1]{So Yamagata\thanks{so.yamagata@fukuoka-u.ac.jp}}
\affil[1]{Department of Applied Mathematics, Faculty of Science, Fukuoka University}
\date{}

\begin{document}

\maketitle
\begin{abstract}
    Discrete homotopy theory or $A$-homotopy theory is a combinatorial homotopy theory defined on graphs, simplicial complexes, and metric spaces, reflecting information about their connectivity.
    The present paper aims to further understand the (non-)similarities between the $A$-homotopy and ordinary homotopy theories through explicit constructions.
    More precisely, we define mapping fiber graphs and study their basic properties yielding, under a technical condition, a discrete analogous of Puppe sequence in a naive discrete homotopy theory.
\end{abstract}

\renewcommand{\thefootnote}{}
\footnotetext{
    \noindent
    \textit{2020 Mathematics Subject Classification}: 05C25(primary), 55P10(secondary)}
\footnotetext{
    \textit{Key words}: discrete homotopy theory, graph theory, mapping fiber
}\renewcommand{\thefootnote}{\arabic{footnote}}

\section{Introduction}
Discrete homotopy theory or $A$-homotopy theory is a combinatorial homotopy theory defined on graphs, simplicial complexes, and metric spaces, reflecting information about their connectivity.
The origins of discrete homotopy theory can be traced back to the pioneering work of physicist Atkin \cite{Atka}, \cite{Atkb}, who utilized simplicial complexes to model social networks, thereby capturing the essence of connectivity within partial networks.

In \cite{BKLW}, the foundational work of the $A$-homotopy groups defined on simplicial complexes and graphs was introduced, demonstrating that these groups are isomorphic.
This isomorphism was established through the association of ``connectivity graphs,'' which were constructed from the faces of a simplicial complex, where vertices represent the faces and edges connect faces that share a common lower-dimensional face.
This construction brings us back to the study of $A$-homotopy groups of graphs.

The development of $A$-homotopy theory for graphs was further advanced in \cite{BBLL}, which associated a cubical set with a graph and examined the homotopy group of this cubical set.
They conjectured that under certain conditions, called the cubical approximation property, the $A$-homotopy group of a graph would be isomorphic to the homotopy group of the geometric realization of the associated cubical sets. This conjecture was confirmed in 2022 by Carranza and Kapulkin \cite{CKa}, \cite{CKb}, who provided a category-theoretic proof that did not rely on the cubical approximation property developing the theory of homotopy theory of cubical sets.

The fundamental properties of $A$-homotopy groups, including the vanishing property \cite{Lb}, lifting property \cite{Mo}, and the relationships between fundamental groupoids and covering graphs \cite{KMb}, have been extensively studied.
Notably, it was shown in \cite{CKc} that colimits do not exist within naive discrete homotopy theory, which was characterized by $A$-homotopy equivalences.
The $(\infty, 1)$-category-theoretic construction of discrete homotopy theory was introduced in \cite{CKL}.
More recently, the fibration structure in the category of graphs has been studied in \cite{KMa}.

From a broader perspective, associated homology theories have also been developed in the context of $A$-homotopy theory.
In \cite{BCW}, the notion of $A$-homotopy in metric spaces and its associated ``discrete (cubical) homology'' groups were introduced.
The relationship between these homology groups and the $A$-homotopy groups of metric spaces has been explored, contributing to the growing body of literature on discrete (cubical) homology theories (see \cite{BCW}, \cite{BGJWa}, \cite{BGJWb}, \cite{BGJWc} and \cite{KK} for the literature on the ``discrete (cubical) homology theories'').

The potential applications of $A$-homotopy theory are significant and varied.
For instance, in \cite{BS}, a formula for the Betti number of the complement of 3-equal arrangements was derived combinatorially by computing the $A$-fundamental group for ordered complexes of the Boolean lattices.
Additionally, \cite{BSW} examined the relationship between the homotopy group of the complement of the $k$-parabolic subspace arrangement of the real reflection group and the discrete homotopy group of the Coxeter complex.
In particular, it gives another proof of the $K(\pi, 1)$-ness of the complement of the real 3-equal subspace arrangement proved by Khovanov \cite{K}.

The primary purpose of the present paper is to enhance the understanding of the (non-)similarities between $A$-homotopy and classical homotopy theories through explicit constructions.
Specifically, we will define the notion of a mapping fiber within the category of graphs equipped with $A$-homotopies and investigate its basic properties.

The organization of the present paper is as follows:
Section 2 reviews the basic notions and properties of $A$-homotopy theory and defines the naive discrete homotopy theory.
Section 3 introduces the mapping fiber graph and studies its basic properties, under a technical condition, yielding a discrete analogous Puppe sequence in a naive discrete homotopy theory.
\paragraph*{Acknowledgement}
The author would like to thank Takuya Saito for many helpful discussions. This work was supported by JSPS KAKENHI Grant Number JP24K16926.
\section{Discrete homotopy theory of graphs}
In this section, we review the discrete homotopy theory of graphs.
For the material, we primarily refer to \cite{BBLL} and \cite{CKa} (see also \cite{BKLW} for the original treatment and \cite{BL} for an overview of the theory).

We assume a graph $G = (V_G, E_G)$ with vertex set $V_G$ and (undirected) edge set $E_G$ to be simple, i.e., with no loops and no multiple edges.
\begin{defn}
    \begin{itemize}
        \item For a graph $G = (V_G, E_G)$, if two vertices $u_1, u_2 \in V_G$ are connected by an edge, i.e., $u_1u_2 \in E_G$ we say that $u_1$ and $u_2$ are \textbf{adjacent} and write $u_1 \sim u_2$.
        \item For graphs $G = (V_G, E_G)$, $H = (V_H, E_H)$, the \textbf{Cartesian product} $G \otimes H$ is the graph $(V_{G \otimes H}, E_{G \otimes H})$ defined as follows.
              The vertex set is defined by $V_{G \otimes H} = V_G \times V_H$, and for two vertices $(u_1, v_1), (u_2, v_2)$ $\in V_{G \otimes H}$, $(u_1, v_1) \sim (u_2, v_2)$ if $u_1 = u_2$ and $v_1 \sim v_2$, or $u_1 \sim u_2$ and $v_1 = v_2$.
        \item For graphs $G = (V_G, E_G)$, $H = (V_H, E_H)$,
              a \textbf{graph map} $f: G \to H$ is a set map $V_G \to V_H$ such that if $u_1 \sim u_2$, then either $f(u_1) = f(u_2)$ or $f(u_1) \sim f(u_2)$.
    \end{itemize}
\end{defn}
In the rest, we write $u \in G$ instead of $u \in V_G$ as an abuse of notation.
\begin{defn}
    \begin{itemize}
        \item For an integer $m \geq 0$, define $I_m$ as a graph with $m+1$ vertices labeled $0, 1, \dots, m$ and $m$ edges $i \ i + 1$ for $i = 0, 1, \dots, m - 1$.
              We allow $m$ to be (countable) infinite.
              In this case, $I_\infty$ is the graph with vertices labeled $i \in \Zt$ and edges $i \ i + 1$, for $i \in \Zt$.
        \item For $n, M \geq 0$, let $I_{\geq M}^{\otimes n}$ denote an induced subgraph of $I_{\infty}^{\otimes n}$ with vertices $(u_1, u_2, \dots, u_n)$, where $|u_i| \geq M$ for some $i$ (see Figures \ref{fig:I_infty}, \ref{fig:I_geqM}).
    \end{itemize}
\end{defn}
\begin{figure}[h]
    \centering
    \begin{minipage}[b]{0.45\linewidth}
        \centering
        \raisebox{0.3cm}{
            \begin{tikzpicture}
                \fill (0,0) circle (1.5pt);
                \fill (0,0.5) circle (1.5pt);
                \fill (0,1) circle (1.5pt);
                \fill (0,1.5) circle (1.5pt);
                \fill (0,-0.5) circle (1.5pt);
                \fill (0,-1) circle (1.5pt);
                \fill (0,-1.5) circle (1.5pt);
                \fill (0.5,0) circle (1.5pt);
                \fill (0.5,0.5) circle (1.5pt);
                \fill (0.5,1) circle (1.5pt);
                \fill (0.5,1.5) circle (1.5pt);
                \fill (0.5,-0.5) circle (1.5pt);
                \fill (0.5,-1) circle (1.5pt);
                \fill (0.5,-1.5) circle (1.5pt);
                \fill (1,0) circle (1.5pt);
                \fill (1,0.5) circle (1.5pt);
                \fill (1,1) circle (1.5pt);
                \fill (1,1.5) circle (1.5pt);
                \fill (1,-0.5) circle (1.5pt);
                \fill (1,-1) circle (1.5pt);
                \fill (1,-1.5) circle (1.5pt);
                \fill (1.5,0) circle (1.5pt);
                \fill (1.5,0.5) circle (1.5pt);
                \fill (1.5,1) circle (1.5pt);
                \fill (1.5,1.5) circle (1.5pt);
                \fill (1.5,-0.5) circle (1.5pt);
                \fill (1.5,-1) circle (1.5pt);
                \fill (1.5,-1.5) circle (1.5pt);
                \fill (-0.5,0) circle (1.5pt);
                \fill (-0.5,0.5) circle (1.5pt);
                \fill (-0.5,1) circle (1.5pt);
                \fill (-0.5,1.5) circle (1.5pt);
                \fill (-0.5,-0.5) circle (1.5pt);
                \fill (-0.5,-1) circle (1.5pt);
                \fill (-0.5,-1.5) circle (1.5pt);
                \fill (-1,0) circle (1.5pt);
                \fill (-1,0.5) circle (1.5pt);
                \fill (-1,1) circle (1.5pt);
                \fill (-1,1.5) circle (1.5pt);
                \fill (-1,-0.5) circle (1.5pt);
                \fill (-1,-1) circle (1.5pt);
                \fill (-1,-1.5) circle (1.5pt);
                \fill (-1.5,0) circle (1.5pt);
                \fill (-1.5,0.5) circle (1.5pt);
                \fill (-1.5,1) circle (1.5pt);
                \fill (-1.5,1.5) circle (1.5pt);
                \fill (-1.5,-0.5) circle (1.5pt);
                \fill (-1.5,-1) circle (1.5pt);
                \fill (-1.5,-1.5) circle (1.5pt);
                \draw (-1.75,0) -- (1.75,0);
                \draw (-1.75,0.5) -- (1.75,0.5);
                \draw (-1.75,1) -- (1.75,1);
                \draw (-1.75,1.5) -- (1.75,1.5);
                \draw (-1.75,-0.5) -- (1.75,-0.5);
                \draw (-1.75,-1) -- (1.75,-1);
                \draw (-1.75,-1.5) -- (1.75,-1.5);
                \draw (0,-1.75) -- (0,1.75);
                \draw (0.5,-1.75) -- (0.5,1.75);
                \draw (1,-1.75) -- (1,1.75);
                \draw (1.5,-1.75) -- (1.5,1.75);
                \draw (-0.5,-1.75) -- (-0.5,1.75);
                \draw (-1,-1.75) -- (-1,1.75);
                \draw (-1.5,-1.75) -- (-1.5,1.75);
                \node at (0, -1.9) {$\vdots$}; \node at (0.5, -1.9) {$\vdots$}; \node at (1, -1.9) {$\vdots$}; \node at (1.5, -1.9) {$\vdots$}; \node at (-0.5, -1.9) {$\vdots$}; \node at (-1, -1.9) {$\vdots$}; \node at (-1.5, -1.9) {$\vdots$}; \node at (0, 2.1) {$\vdots$}; \node at (0.5, 2.1) {$\vdots$}; \node at (1, 2.1) {$\vdots$}; \node at (1.5, 2.1) {$\vdots$}; \node at (-0.5, 2.1) {$\vdots$}; \node at (-1, 2.1) {$\vdots$}; \node at (-1.5, 2.1) {$\vdots$};
                \node at (-2, 0) {$\dots$}; \node at (-2, 0.5) {$\dots$}; \node at (-2, 1) {$\dots$}; \node at (-2, 1.5) {$\dots$}; \node at (-2, -0.5) {$\dots$}; \node at (-2, -1) {$\dots$}; \node at (-2, -1.5) {$\dots$}; \node at (2.1, 0) {$\dots$}; \node at (2.1, 0.5) {$\dots$}; \node at (2.1, 1) {$\dots$}; \node at (2.1, 1.5) {$\dots$}; \node at (2.1, -0.5) {$\dots$}; \node at (2.1, -1) {$\dots$}; \node at (2.1, -1.5) {$\dots$};
            \end{tikzpicture}}\caption{The graph $I_{\infty}^{\otimes 2}$.}\label{fig:I_infty}

    \end{minipage}
    \begin{minipage}[b]{0.45\linewidth}
        \centering
        \begin{tikzpicture}
            \fill (0,1) circle (1.5pt);
            \fill (0,1.5) circle (1.5pt);
            \fill (0,-1) circle (1.5pt);
            \fill (0,-1.5) circle (1.5pt);
            \fill (0.5,1) circle (1.5pt);
            \fill (0.5,1.5) circle (1.5pt);
            \fill (0.5,-1) circle (1.5pt);
            \fill (0.5,-1.5) circle (1.5pt);
            \fill (1,0) circle (1.5pt);
            \fill (1,0.5) circle (1.5pt);
            \fill (1,1) circle (1.5pt);
            \fill (1,1.5) circle (1.5pt);
            \fill (1,-0.5) circle (1.5pt);
            \fill (1,-1) circle (1.5pt);
            \fill (1,-1.5) circle (1.5pt);
            \fill (1.5,0) circle (1.5pt);
            \fill (1.5,0.5) circle (1.5pt);
            \fill (1.5,1) circle (1.5pt);
            \fill (1.5,1.5) circle (1.5pt);
            \fill (1.5,-0.5) circle (1.5pt);
            \fill (1.5,-1) circle (1.5pt);
            \fill (1.5,-1.5) circle (1.5pt);
            \fill (-0.5,1) circle (1.5pt);
            \fill (-0.5,1.5) circle (1.5pt);
            \fill (-0.5,-1) circle (1.5pt);
            \fill (-0.5,-1.5) circle (1.5pt);
            \fill (-1,0) circle (1.5pt);
            \fill (-1,0.5) circle (1.5pt);
            \fill (-1,1) circle (1.5pt);
            \fill (-1,1.5) circle (1.5pt);
            \fill (-1,-0.5) circle (1.5pt);
            \fill (-1,-1) circle (1.5pt);
            \fill (-1,-1.5) circle (1.5pt);
            \fill (-1.5,0) circle (1.5pt);
            \fill (-1.5,0.5) circle (1.5pt);
            \fill (-1.5,1) circle (1.5pt);
            \fill (-1.5,1.5) circle (1.5pt);
            \fill (-1.5,-0.5) circle (1.5pt);
            \fill (-1.5,-1) circle (1.5pt);
            \fill (-1.5,-1.5) circle (1.5pt);
            \draw (-1.75,0) -- (-1,0);
            \draw (1,0) -- (1.75,0);
            \draw (-1.75,0.5) -- (-1,0.5);
            \draw (1,0.5) -- (1.75,0.5);
            \draw (-1.75,1) -- (1.75,1);
            \draw (-1.75,1.5) -- (1.75,1.5);
            \draw (-1.75,-0.5) -- (-1,-0.5);
            \draw (1,-0.5) -- (1.75,-0.5);
            \draw (-1.75,-1) -- (1.75,-1);
            \draw (-1.75,-1.5) -- (1.75,-1.5);
            \draw (0,-1.75) -- (0,-1);
            \draw (0,1) -- (0,1.75);
            \draw (0.5,-1.75) -- (0.5,-1);
            \draw (0.5,1) -- (0.5,1.75);
            \draw (1,-1.75) -- (1,1.75);
            \draw (1.5,-1.75) -- (1.5,1.75);
            \draw (-0.5,-1.75) -- (-0.5,-1);
            \draw (-0.5,1) -- (-0.5,1.75);
            \draw (-1,-1.75) -- (-1,1.75);
            \draw (-1.5,-1.75) -- (-1.5,1.75);
            \node at (0, -2.4) {$\vdots$}; \node at (0.5, -1.9) {$\vdots$}; \node at (1, -2.4) {$\vdots$}; \node at (1.5, -1.9) {$\vdots$}; \node at (-0.5, -1.9) {$\vdots$}; \node at (-1, -2.4) {$\vdots$}; \node at (-1.5, -1.9) {$\vdots$}; \node at (0, 2.1) {$\vdots$}; \node at (0.5, 2.1) {$\vdots$}; \node at (1, 2.1) {$\vdots$}; \node at (1.5, 2.1) {$\vdots$}; \node at (-0.5, 2.1) {$\vdots$}; \node at (-1, 2.1) {$\vdots$}; \node at (-1.5, 2.1) {$\vdots$};
            \node at (-2, 0) {$\dots$}; \node at (-2, 0.5) {$\dots$}; \node at (-2, 1) {$\dots$}; \node at (-2, 1.5) {$\dots$}; \node at (-2, -0.5) {$\dots$}; \node at (-2, -1) {$\dots$}; \node at (-2, -1.5) {$\dots$}; \node at (2.6, 0) {$\dots$}; \node at (2.1, 0.5) {$\dots$}; \node at (2.6, 1) {$\dots$}; \node at (2.1, 1.5) {$\dots$}; \node at (2.1, -0.5) {$\dots$}; \node at (2.6, -1) {$\dots$}; \node at (2.1, -1.5) {$\dots$};
            \node at (2.1,0) {0};
            \node at (2.1,1) {$M$};
            \node at (2.1,-1) {$-M$};
            \node at (0,-2) {0};
            \node at (1,-2) {$M$};
            \node at (-1,-2) {$-M$};
        \end{tikzpicture}\caption{The graph $I_{\geq M}^{\otimes 2}$.}\label{fig:I_geqM}
    \end{minipage}
\end{figure}
We define an equivalence relation on the set of graph maps by the following.
\begin{defn}
    Two graph maps $f, g: G \to H$ are called \textbf{$A$-homotopic} and denoted by $f \simeq g$ if there exists an integer $m \geq 0$ and a graph map
    \begin{equation*}
        F : G \otimes I_m \to H
    \end{equation*}
    such that $F(\bullet, 0) = f$, $F(\bullet, m) = g$.
    We call the map $F$ an \textbf{$A$-homotopy from $f$ to $g$}.
\end{defn}

The $A$-homotopy equivalence measures a ``combinatorial'' invariant of graph maps that differs from the ordinary homotopy; in the $A$-homotopy sense, it is not a topological invariant of graphs regarded as 1-dimensional simplicial complexes.
One typical example showing the difference between the $A$-homotopy and the ordinary homotopy is the cyclic graph $C_n$.
When cyclic graphs are regarded as one-dimensional complexes, they are homotopy equivalent to a circle $S^1$.
On the other hand, if $n = 3, 4$, the graph $C_n$ is $A$-homotopy equivalent to a single vertex in the sense that there exist $A$-homotopies
\begin{equation*}
    F: C_n \otimes I_m \to C_n
\end{equation*}
such that $F|_{C_n \otimes \{0\}} = \id_{C_n}$ and $F|_{C_n \otimes \{ m \}} = c_{0}$ for some $m$, where $c_0$ is the constant graph map sending all vertices in $C_n$ to $0 \in C_n$ (see Figures \ref{fig:C3} and \ref{fig:C4}, where the left labeled vertices are sent to the right corresponding vertices). On the other hand, when $n \geq 5$, the graph $C_n$ is not $A$-homotopy equivalent to a single vertex in this sense.

\begin{figure}[h]
    \begin{minipage}[b]{0.48\columnwidth}
        \centering
        \begin{tikzpicture}
            \fill (0,0) circle (1.5pt);
            \fill (1.6,0) circle (1.5pt);
            \fill (0.8,0.8) circle (1.5pt);
            \fill (0,1.5) circle (1.5pt);
            \fill (1.6,1.5) circle (1.5pt);
            \fill (0.8,2.3) circle (1.5pt);
            \fill (3.5,0.5) circle (1.5pt);
            \fill (5.1,0.5) circle (1.5pt);
            \fill (4.3,1.8) circle (1.5pt);

            \draw (0,0) -- (1.6,0) -- (0.8,0.8) --cycle;
            \draw (0,1.5) -- (1.6,1.5) -- (0.8,2.3) -- cycle;
            \draw (0,0) -- (0,1.5) -- (1.6,1.5) -- (1.6,0) -- cycle;
            \draw (0.8,0.8) -- (0.8,1.4);
            \draw (0.8,1.6) -- (0.8,2.3);
            \draw[->] (2,1) -- (3,1);
            \draw (3.5,0.5) -- (5.1,0.5) -- (4.3,1.8) -- cycle;
            \draw[->] (1.6, 2) .. controls (2.6,2.6) .. (3.5, 1.7);
            \draw[->] (1.6, -0.5) .. controls (2.6,-1.1) .. (3.5, -0.2);

            \node at (-0.2, 0) {0};
            \node at (1.8, 0) {1};
            \node at (1, 0.9) {2};
            \node at (-0.2, 1.5) {0};
            \node at (1.8, 1.5) {0};
            \node at (0.8, 2.5) {0};
            \node at (3.5, 0.2) {0};
            \node at (5.1, 0.2) {1};
            \node at (4.3, 2.1) {2};
            \node at (2.6,2.6) {$c_0$};
            \node at (2.6,-1.2) {$\id_{C_3}$};
            \node at (2.5, 1.3) {$F$};
        \end{tikzpicture}\caption{$A$-homotopy $C_3 \otimes I_1 \to C_3$ between $c_0$ and $\id_{C_3}$.}\label{fig:C3}
    \end{minipage} \quad
    \begin{minipage}[b]{0.48\columnwidth}
        \centering
        \begin{tikzpicture}
            \fill (0,0) circle (1.5pt);
            \fill (1.6,0) circle (1.5pt);
            \fill (2,1) circle (1.5pt);
            \fill (0.4,1) circle (1.5pt);
            \fill (0,1.5) circle (1.5pt);
            \fill (1.6,1.5) circle (1.5pt);
            \fill (2,2.5) circle (1.5pt);
            \fill (0.4,2.5) circle (1.5pt);
            \fill (0,3) circle (1.5pt);
            \fill (1.6,3) circle (1.5pt);
            \fill (2,4) circle (1.5pt);
            \fill (0.4,4) circle (1.5pt);
            \fill (4,1.2) circle (1.5pt);
            \fill (5.6,1.2) circle (1.5pt);
            \fill (5.6,2.8) circle (1.5pt);
            \fill (4,2.8) circle (1.5pt);

            \draw (0.4,1) -- (0,0) -- (1.6,0) -- (2,1);
            \draw (0.4,2.5) -- (0,1.5) -- (1.6,1.5) -- (2,2.5);
            \draw (0,3) -- (1.6,3) -- (2,4) -- (0.4,4) -- cycle;
            \draw (0,0) -- (0,1.5) -- (0,3);
            \draw (1.6,0) -- (1.6,1.5) -- (1.6,3);
            \draw (2,1) -- (2,2.5) -- (2,4);
            \draw (0.4,1) -- (1.5,1);
            \draw (1.7,1) -- (2,1);
            \draw (0.4,2.5) -- (1.5,2.5);
            \draw (1.7,2.5) -- (2,2.5);
            \draw (0.4,1) -- (0.4,1.4);
            \draw (0.4,1.6) -- (0.4,2.5);
            \draw (0.4,2.5) -- (0.4,2.9);
            \draw (0.4,3.1) -- (0.4,4);
            \draw[->] (2.5,2) -- (3.5,2);
            \draw (4,1.2) -- (5.6,1.2) -- (5.6,2.8) -- (4,2.8) -- cycle;
            \draw[->] (2.5, 4) .. controls (3.4,4) .. (4, 3.5);
            \draw[->] (2.5, 0) .. controls (3.4,0) .. (4, 0.5);

            \node at (-0.2, 0) {0};
            \node at (1.8,0) {1};
            \node at (2.2,1) {2};
            \node at (0.2,1) {3};
            \node at (-0.2,1.5) {0};
            \node at (1.8,1.5) {0};
            \node at (2.2,2.5) {3};
            \node at (0.2,2.5) {3};
            \node at (-0.2,3) {0};
            \node at (1.8,3) {0};
            \node at (2.2,4) {0};
            \node at (0.2,4) {0};
            \node at (3,2.3) {$F$};
            \node at (4,0.9) {0};
            \node at (5.6,0.9) {1};
            \node at (5.6,3.1) {2};
            \node at (4,3.1) {3};
            \node at (3.4,4.1) {$c_0$};
            \node at (3.5,-0.3) {$\id_{C_4}$};
        \end{tikzpicture}\caption{$A$-homotopy $C_4 \otimes I_1 \to C_4$ between $c_0$ and $\id_{C_4}$.}\label{fig:C4}
    \end{minipage}
\end{figure}

Next, we define the category of graphs, pairs of graphs, and $A$-homotopy groups.
\begin{defn}
    \begin{itemize}
        \item Define $\grh$ as the category of graphs as objects and graph maps as morphisms.
        \item Define $\grh^2$ as the category defined by the following.
              The object consists of ordered pairs $(G, A)$ of graphs, where $A$ is a subgraph of $G$.
              A morphism $f: (G, A) \to (H, B)$, called a \textbf{relative graph map},
              is a graph map $f: G \to H$ whose restriction to $A$ is a graph map $f|_A: A \to B$.
        \item Define $\grh_*$ as the subcategory of $\grh^2$ with the object consisting of all pairs $(G, u_0)$, where $u_0$ is a vertex of $G$.
              The pair $(G, u_0)$ is called a \textbf{pointed graph}, and $u_0$ is called the \textbf{base vertex}.
              The morphism in $\grh_*$ is called a \textbf{pointed graph map}.
              Sometimes we denote the pointed graph $(G, u_0)$ by $G$ and morphism $f: (G, u_0) \to (H, v_0)$ by $f: G \to H$ as abuse of notation.
              Similarly, we can define the category $\sets_*$.
    \end{itemize}
\end{defn}

\begin{defn}
    Two relative graph maps $f, g: (G, A) \to (H, B)$ are said to be \textbf{relatively $A$-homotopic} and denoted by $f \simeq g$ if there exists an integer $m \geq 0$ and a graph map
    \begin{equation*}
        F : G \otimes I_m \to H
    \end{equation*}
    such that $F(\bullet, 0) = f$, $F(\bullet, m) = g$, and $F|_{A \otimes I_m} : A \otimes I_m \to B$ is an $A$-homotopy from $f|_A$ to $g|_A$.
    We call the map \textbf{relative $A$-homotopy from $(G, A)$ to $(H, B)$}.
\end{defn}
Now, we define the $A$-homotopy groups of a graph.
\begin{defn}
    Let $n \geq 0$ and $u_0 \in G$.
    The \textbf{$n$-th $A$-homotopy group of $G$ at $u_0$} is the set
    \begin{equation*}
        A_n (G, u_0) = \{ (I_\infty^{\otimes n}, I_{\geq M}^{\otimes n}) \to (G, u_0) \mid M \geq 0 \} / \simeq
    \end{equation*}
    where $\simeq$ is defined by the relative $A$-homotopy.
\end{defn}
\begin{rem}
    For $n = 0$, $A_0(G, u_0)$ is just a pointed set of connected components of $G$, with a distinguished component containing $u_0$.
    For $n \geq 1$, $A_n(G, u_0)$ has a group structure, and it is abelian when $n \geq 2$ (for the group structure, see Proposition 3.5 of \cite{BKLW} for instance).
\end{rem}
Although the $A$-homotopy group is an invariant measuring a quantity (often called a ``combinatorial hole'') that is different from the geometric homotopy group of the graph, it is known that there exists a geometric model, a particular space associated with the graph, whose homotopy group is isomorphic to the $A$-homotopy group of the graph. The first result is due to Barcelo--Kramer--Laubenbacher--Weaver \cite{BKLW}.
\begin{thm}[\cite{BKLW}]\label{thm:fund}
    For a graph $G$, define $X_G$ as the topological space obtained from $G$ regarded as a one-dimensional cell complex by attaching 2-cells along the boundary of each 3- and 4-cycle.
    Then,
    \begin{equation*}
        A_1(G, u_0) \simeq \pi_1(X_G, u_0).
    \end{equation*}
\end{thm}
The following theorem gives a generalization of Theorem \ref{thm:fund}, which was conjectured by Babson--Barcelo--Longueville--Laubenbacher \cite{BBLL} and proved by Carranza--Kapulkin \cite{CKa}.
\begin{thm}[\cite{CKa}]
    For a graph $G$, associate a cubical set $X_G$ and let $|X_G|$ be its geometric realization.
    Then,
    \begin{equation*}
        A_n(G, u_0) \simeq \pi_n(|X_G|, u_0)
    \end{equation*}
    for all $n$.
\end{thm}
Note that, in \cite{BBLL}, the theorem assumed a plausible property called the cubical approximation property, which is still open. On the other hand, the theorem was proved in \cite{CKa} without using the cubical approximation property but by a category-theoretic approach.

Now, we define a naive discrete homotopy category of graphs.
\begin{defn}
    A \textbf{congruence} on a category $\mathcal{C}$ is an equivalence relation $\sim$ on the class $\bigcup_{A, B} \Hom(A,B)$ of all morphisms in $\mathcal{C}$ such that
    \begin{itemize}
        \item $f \in \Hom(A, B)$ and $f \sim f'$ imply $f' \in \Hom(A, B)$;
        \item $f \sim f'$, $g \sim g'$ and $g \circ f$ exists imply that $g \circ f \sim g' \circ f'$.
    \end{itemize}
\end{defn}
\begin{thm}\label{thm:quotient_cat}
    Let $\mathcal{C}$ be a category with congruence $\sim$, and let $[f]$ denote the equivalence class of a morphism $f$.
    Then, $\mathcal{C'}$ defined by
    \begin{itemize}
        \item $\obj \mathcal{C'} = \obj \mathcal{C}$;
        \item $\Hom_{\mathcal{C'}} (A, B) = \{ [f] \mid f \in \Hom_{\mathcal{C}}(A, B) \}$;
        \item $[g] \circ [f] = [g \circ f]$
    \end{itemize}
    is a category.
    The category $\mathcal{C}'$ is called a \textbf{quotient category}.
\end{thm}
We denote $\Hom_{\mathcal{C'}} (A, B)$ by $[A, B]$.
\begin{defn}
    We refer to the quotient category of $\grh$ by $A$-homotopy as the \textbf{naive discrete homotopy category} of graphs\footnote{As noted in the introduction of \cite{CKc}, we can work with not only the $A$-homotopy but also the weak $A$-homotopy, i.e., maps inducing isomorphisms on all $A$-homotopy groups.
        The authors of that work referred to the theory with the former $A$-homotopy as the naive discrete homotopy theory and that with the latter homotopy as the discrete homotopy theory.
        Since we are working with $A$-homotopy in the present paper, we adopt the term, a naive discrete homotopy theory following their usage.} and denote it by $\hgrh$.
    If the objects are pointed graphs, we denote the category by $\hgrh_*$.
\end{defn}
\section{Mapping fiber graphs}
In this section, we introduce the mapping fiber graphs and study their basic properties with reference to \cite{R}.
We assume the graphs $I_m$, $m \in \mathbb{N} \cup \{ 0, \infty \}$ are pointed at 0.
Let us begin with defining the exactness of sequences in $\hgrh_*$.
\begin{defn}
    A sequence of pointed graphs and pointed graph maps
    \begin{equation*}
        \cdots \to G_{n + 1} \xrightarrow{f_{n + 1}} G_n \xrightarrow{f_n} G_{n - 1} \to \cdots
    \end{equation*}
    is said to be \textbf{exact} in $\hgrh_*$ if for any pointed graph $H$, the induced sequence
    \begin{equation*}
        \cdots \to [H, G_{n + 1}] \xrightarrow{f_{n + 1, *}} [H, G_n] \xrightarrow{f_{n, *}} [H, G_{n - 1}] \to \cdots
    \end{equation*}
    is exact in $\sets_*$.
\end{defn}
Let us review the definition of the path graph.
\begin{defn}\label{def:path}
    For a pointed graph $G$, we call the pointed graph map $\omega: I_\infty \to G$ a \textbf{(pointed) path of length $l$} if it satisfies the condition
    \begin{equation*}
        \omega(t) = \begin{cases} u_0, & t \leq 0 \\ \omega(l), & t \geq l \end{cases}
    \end{equation*}
    for some $l \in \mathbb{N} \cup \{ 0 \}$ (see Figure \ref{fig:omega}).
\end{defn}
\begin{figure}[h]
    \centering
    \begin{tikzpicture}
        \fill (-2,0) circle (3pt);
        \fill (0,0) circle (3pt);
        \fill (2,0) circle (3pt);
        \fill (5,0) circle (3pt);
        \fill (7,0) circle (3pt);
        \draw (-2.5,0) -- (-2,0);
        \draw (-2,0) -- (0,0);
        \draw (0,0) -- (2,0);
        \draw (2,0) -- (3.5,0);
        \draw (4.5,0) -- (5,0);
        \draw (5,0) -- (7,0);
        \draw (7,0) -- (7.5,0);
        \node at (-3,0) {$\cdots$};
        \node at (4,0) {$\cdots$};
        \node at (8,0) {$\cdots$};
        \node at (-2,-0.4) {$-1$};
        \node at (0,-0.4) {$0$};
        \node at (2,-0.4) {$1$};
        \node at (5,-0.4) {$l$};
        \node at (7,-0.4) {$l + 1$};
        \node at (-2,0.4) {$u_0$};
        \node at (0,0.4) {$u_0$};
        \node at (2,0.4) {$\omega(1)$};
        \node at (5,0.4) {$\omega(l)$};
        \node at (7,0.4) {$\omega(l)$};
    \end{tikzpicture}\caption{The path $\omega$ of length $l$.}\label{fig:omega}
\end{figure}

\begin{defn}\label{def:path}
    Let $G$ be a pointed graph.
    The \textbf{path graph} $PG$ is the graph on the vertex set
    \begin{equation*}
        V_{PG} = \left\{ \omega \ \middle| \ l \in \mathbb{N} \cup \{ 0 \}, \omega: I_\infty \to G \text{ a path of length } l \right\}.
    \end{equation*}
    Two vertices $\omega, \omega': I_\infty \to G$ are adjacent if there exists a graph map
    \begin{equation*}
        F: I_{\infty} \otimes I_1 \to G
    \end{equation*}
    such that $F(\bullet, 0) = \omega$ and $F(\bullet, 1) = \omega'$.
    The base vertex of $PG$ is chosen as $\omega_0: I_\infty \to G$ a path of length 0.
\end{defn}
\begin{defn}\label{def:loop}
    Let $G$ be a pointed graph.
    The \textbf{loop graph} $\Omega G$ is the induced subgraph of $PG$ on the vertices $\omega$ paths of length $l$ with $\omega(l) = u_0$, $l \in \mathbb{N} \cup \{ 0 \}$.
    If $\omega: I_\infty \to G$ is a path of length $l$ with $\omega(l) = u_0$, then we call $\omega$ a \textbf{loop of length $l$}.
    The base vertex of $\Omega G$ is $\omega_0$, where $\omega_0: I_\infty \to H$ is a loop of length 0.
\end{defn}
For a loop $\omega$ of length $l$ and some $0 \leq i_1 < i_2 \leq l$, consider a graph map $\tilde{\omega}: I_{[i_1, i_2]} \to G$, where the graph $I_{[i_1, i_2]}$ is the induced subgraph of $I_\infty$ consisting of the vertices $i_1, i_1 + 1, \dots, i_2$, and the map is obtained by restricting on that vertices.
We call $\tilde{\omega}$ a subloop of $\omega$ if it satisfies
\begin{equation*}
    \tilde{\omega}(t) = \begin{cases}
        u_0                    & \text{if $t = i_1, i_2$}   \\
        \omega(t) \ (\neq u_0) & \text{if $i_1 < t < i_2$},
    \end{cases}
\end{equation*}
and call the number $l_{\tilde{\omega}} \coloneqq | \{ \tilde{\omega}(t) \mid i_1 \leq t \leq i_2 \}|$ a sublength of $\omega$.

For a pointed graph map $f_1 : G \to H$, define a map $\Omega f_1: \Omega G \to \Omega H$ by $\omega \to f_1 \omega$.
Now, let us define the mapping fiber graph.
\begin{defn}\label{def:mf}
    For a pointed graph map $f_1: G \to H$ define the \textbf{mapping fiber} graph $Mf_1$ as the graph on the vertex set
    \begin{equation*}
        V_{Mf_1} = \left\{ (u, \omega) \in G \otimes PH \ \middle| \ l \in \mathbb{N} \cup \{ 0 \}, \omega \text{ a path of length $l$ with } \omega(l) = f_1(u)  \right\}.
    \end{equation*}
    For the base vertex of $Mf_1$ we choose $(u_0, \omega_0)$, where $u_0 \in G$ is a base vertex of $G$ and $\omega_0: I_\infty \to H$ is a path of length 0.
    By definition of the Cartesian product of graphs, we can see that two vertices $(u, \omega), (u', \omega') \in Mf_1$ are adjacent, i.e., $(u, \omega) \sim (u', \omega')$ if $u = u'$ and $\omega \sim \omega'$ or $u \sim u'$ and $\omega = \omega'$.
\end{defn}
Equivalently, the graph $Mf_1$ can be defined as the pullback diagram
\begin{equation*}
    \begin{tikzcd}
        Mf_1 \arrow[r] \arrow[d] \arrow[dr, phantom, "\lrcorner", very near start] & PH \arrow[d, "p"] \\
        G \arrow[r, "f_1", swap] & H \\
    \end{tikzcd}
\end{equation*}
where $p: PH \to H$ is a projection map defined for $\omega$ a path of length $l$ by $p_1(\omega) = \omega(l)$.

For a pointed graph map $f_1: G \to H$, we have a graph monomorphism $k: \Omega H \to Mf_1$ defined by $\omega \mapsto (u_0, \omega)$ and a projection $f_2: Mf_1 \to G$ defined by $(u, \omega) \to u$.
Both maps $k$ and $f_2$ are pointed graph maps and thus yield the sequence of pointed graph maps
\begin{equation*}
    \Omega G \xrightarrow{\Omega f_1} \Omega H \xrightarrow{k} Mf_1 \xrightarrow{f_2} G \xrightarrow{f_1} H.
\end{equation*}
\begin{defn}
    For pointed graphs $G, H$ define the graph $G^H$ as the graph on the vertex set
    \begin{equation*}
        V_{G^H} = \left\{ \omega: H \to G \mid \omega \text{ a pointed graph map} \right\}
    \end{equation*}
    with base vertex $\omega_0$ a constant graph map $H \to G$ sending all vertices of $H$ to a base vertex $u_0 \in G$.
    Define the \textbf{evaluation map} $e: G^H \otimes H \to G$ by
    \begin{equation*}
        e(\omega, v) = \omega(v).
    \end{equation*}
\end{defn}
The following proposition holds.
\begin{prop}\label{prop:evaluation}
    Let $G, H, K$ be pointed graphs.
    Then,
    \begin{enumerate}[label=(\arabic*)]
        \item The evaluation map $e: G^H \otimes H \to G$ is a pointed graph map. \label{a1}
        \item A map $\varphi: K \otimes H \to G$ is a pointed graph map if and only if $\varphi^{\#}: K \to G^H$ is a pointed graph map. \label{a2}
    \end{enumerate}
\end{prop}
\begin{proof}
    By construction, it it obvious that the maps preserve base vertices, it suffices to check the maps are graph maps.
    To prove \ref{a1}, take $(\omega, v), (\omega', v') \in G^H \otimes H$ such that $(\omega, v) \sim (\omega', v')$. \\
    If $\omega = \omega'$ and $v \sim v'$, then we have $e(\omega, v) = \omega(v)$ and $e(\omega, v') = \omega(v')$. Since $\omega$ is a graph map, the images are adjacent or coincident; thus, $e$ is a graph map. \\
    If $\omega \sim \omega'$ and $v = v'$, then we have $e(\omega, v) = \omega(v)$ and $e(\omega', v) = \omega'(v)$, which are also adjacent or coincident, and thus, $e$ is a graph map in this case as well. \\
    Next, to prove \ref{a2}, assume that $\varphi: K \otimes H \to G$ is a graph map.
    Define a map $\varphi^\#: K \to G^H$ by $\varphi^\#(w) = \varphi(w, \bullet)$.
    Take $w, w' \in K$ with $w \sim w'$.
    Then, for all $v \in H$,
    \begin{align*}
        \varphi^\#(w) = \varphi(w, v) = \varphi(w', v) = \varphi^\#(w') & \text{ or} \\
        \varphi^\#(w) = \varphi(w, v) \sim \varphi(w', v) = \varphi^\#(w').
    \end{align*}
    holds since $\varphi$ is a graph map.
    Hence, $\varphi^\#$ is a graph map. \\
    Conversely, assume $\varphi^\#: K \to G^H$ is a graph map. Note that $\varphi$ can be factored as
    \begin{equation*}
        K \otimes H \xrightarrow{\varphi^\# \otimes \id} G^H \otimes H \xrightarrow{e} G
    \end{equation*}
    where $e: G^H \otimes H \to G$ is the evaluation map which has been proved to be a graph map.
    It suffices to show that the map $\varphi^\# \otimes \id: K \otimes H \to G^H \otimes H$ is a graph map.
    Take $(w, v), (w', v') \in K \otimes H$ with $(w, v) \sim (w', v')$. \\
    If $w = w'$ and $v \sim v'$, then we have
    \begin{align*}
        (\varphi^\# \otimes \id) (w, v) = (\varphi^\#(w), v) \sim (\varphi^\#(w), v') = (\varphi^\# \otimes \id) (w, v').
    \end{align*}
    On the other hand, assume $w \sim w'$ and $v = v'$. Since $\varphi^{\#}$ is a graph map, $\varphi^{\#}(w) = \varphi^{\#}(w')$ or $\varphi^{\#}(w) \sim \varphi^{\#}(w')$ holds.
    Thus,
    \begin{align*}
        (\varphi^\# \otimes \id) (w, v) = (\varphi^\#(w), v) = (\varphi^\#(w'), v) = (\varphi^\# \otimes \id) (w', v) & \text{ or} \\
        (\varphi^\# \otimes \id) (w, v) = (\varphi^\#(w), v) \sim (\varphi^\#(w'), v) = (\varphi^\# \otimes \id) (w', v)
    \end{align*}
    holds, which implies that $\varphi^\# \otimes \id$ is a graph map.
\end{proof}

\begin{cor}\label{cor:evaluation}
    For pointed graphs $G, H, K$, a map $\psi: K \to G^H$ is a pointed graph map if and only if $e \circ (\psi \otimes \id): K \otimes H \to G^H \otimes H \to G$ is a pointed graph map.
\end{cor}
Note that the exponential law holds.
Consider pointed graphs $G, H, K$. The map $\varphi \mapsto \varphi^{\#}$ gives a bijection $G^{K \otimes H} \to \left( G^H \right)^K$.
Actually, we have an inverse $\left( G^H \right)^K \to G^{K \otimes H}$ defined by $\psi \mapsto \psi^{\flat}$, where $\psi^{\flat}: K \otimes H \to G$ is defined by $\psi^{\flat}(w, v) = \psi(w)(v)$.
As the following theorem shows, taking a loop graph gives a functor in the category $\hgrh_*$.

\begin{thm}\label{thm:loop_functor}
    The loop graph defines an endfunctor $\Omega: \hgrh_* \to \hgrh_*$ on the category $\hgrh_*$.
\end{thm}
\begin{proof}
    By Theorem \ref{thm:quotient_cat}, it suffices to show that there is a functor $\Omega: \grh_* \to \grh_*$ such that if $f_0 \simeq f_1$ then $\Omega f_0 \simeq \Omega f_1$.
    For a pointed graph map $f: G \to H$, define the map $\Omega f: \Omega G \to \Omega H$ by $\Omega f \omega = f \omega$, where $\omega$ is a loop of length $l$, $l \in \mathbb{N} \cup \{ 0 \}$.
    Consider a commutative diagram
    \begin{equation*}
        \xymatrix{
            P_lG \otimes I_l \ar[r]^{f_* \otimes \id} \ar[d]_{e} & P_lH \otimes I_l \ar[d]^{e} \\
            G \ar[r]_{f} & H
        }
    \end{equation*}
    where $P_lG$ and $P_lH$ denote the path graphs associated to graphs $G$ and $H$ consisting of paths of length $l$, respectively, and the two $e$'s are evaluation maps.
    By Proposition \ref{prop:evaluation} \ref{a1}, $e, fe$ are pointed graph maps.
    By Corollary \ref{cor:evaluation}, $e \circ (f_* \otimes \id)$ is a pointed graph map.

    For some $m$ let $F: G \otimes I_m \to H$ be an $A$-homotopy with $F(\bullet, 0) = f_0$ and $F(\bullet, m) = f_1$. Denote the loop graphs obtained as induced subgraphs of $P_lG$ and $P_lH$ by $\Omega_l G$ and $\Omega_l H$, respectively.
    Then the map $\phi: \Omega_l G \otimes I_m \to \Omega_l H$ defined by $\phi (\omega, t) = F_t \omega$ is the desired $A$-homotopy, where $F_t: I_m \to H$ is defined by $F_t(\bullet) = F(\bullet, t)$.
    Actually, we have $\phi (\omega, 0) = F_0 \omega = f_0 \omega$ and $\phi (\omega, m) = F_m \omega = f_1 \omega$.

    We show $\phi$ is a graph map.
    Take vertices $(\omega, t), (\omega', t') \in \Omega_l G \otimes I_m$ with $(\omega, t) \sim (\omega', t')$.
    Assume $\omega = \omega'$ and $t \sim t'$.
    Consider $\phi(\omega, t) = F_{t} \omega$ and $\phi(\omega, t') = F_{t'} \omega$.
    For $s \in I_m$, we have $F_{t} \omega (s) = F(\omega(s), t)$ and $F_{t'} \omega (s) = F(\omega(s), t')$, both of which are coincident or adjacent since $F$ is a graph map.
    Similarly, we can check when $\omega \sim \omega'$, $t = t'$, completes the proof.
\end{proof}

By iterating the construction of the mapping fiber, we can define $Mf_2$ as an induced subgraph of $Mf_1 \otimes PG$.
The vertex set of $Mf_2$ is defined as
\begin{align*}
    V_{Mf_2} = \left\{ (u, \omega, \beta) \in G \otimes PH \otimes PG \ \middle| \ l \in \mathbb{N} \cup \{ 0 \}, \omega, \beta \text{ are paths of length $l$ with } \omega(l) = f_1(u) \text{ in $H$, and} \right. \\
    \left. \beta(l) = f_2(u,\omega) = u \text{ in $G$}\right\}.
\end{align*}
For the base vertex we choose $(u_0, \omega_0, \beta_0)$, where $\omega_0: I_\infty \to H$ and $\beta_0: I_\infty \to G$ are paths of length 0.

Equivalently, the graph $Mf_2$ can be defined as the pullback diagram
\begin{equation*}
    \begin{tikzcd}
        Mf_2 \arrow[r] \arrow[d] \arrow[dr, phantom, "\lrcorner", very near start] & PG \arrow[d, "p_1"] \\
        Mf_1 \arrow[r, "f_2" swap] & G \\
    \end{tikzcd}
\end{equation*}

By definition of the Cartesian product of graphs, two vertices $(u, \omega, \beta), (u', \omega', \beta') \in V_{Mf_2}$ are connected by an edge when $u = u', \omega = \omega'$ and $\beta \sim \beta'$, or $u = u'$, $\omega \sim \omega'$ and $\beta = \beta'$.

\begin{prop}\label{prop:mf2}
    Let $f_1: G \to H$ be a pointed graph map.
    Suppose that for any $\omega \in \Omega G$ and $\beta \in \Omega H$, all of their subloops are of sublength $\leq 4$.
    Then the following diagram commutes in $\hgrh_*$:
    \begin{equation*}
        \xymatrix{
            \Omega G \ar[r]^{\Omega f_1} \ar[d]_{j'} & \Omega H \ar[r]^{k} \ar[d]_{j} & Mf_1 \ar[r]^{f_2} \ar[d]_{\id} & G \ar[r]^{f_1} \ar[d]_{\id} & H \ar[d]_{\id} \\
            Mf_3 \ar[r]_{f_4} & Mf_2 \ar[r]_{f_3} & Mf_1 \ar[r]_{f_2} & G \ar[r]_{f_1} & H.
        }
    \end{equation*}
    The map $j: \Omega H \to Mf_2$ is the graph monomorphism defined by $j(\omega) = (u_0, \omega, \beta_0)$ for a path $\omega: I_\infty \to H$ of length $l$, $l \in \mathbb{N} \cup \{ 0 \}$ and a path $\beta_0: I_\infty \to G$ of length 0.
    The map $j': \Omega G \to Mf_3$ is defined for a path $\beta: I_\infty \to G$ of length $l$, $l \in \mathbb{N} \cup \{ 0 \}$ by $j'(\beta) = (u_0, \omega_0, \beta, \gamma_0)$, where $\omega_0: I_\infty \to H$ and $\gamma_0: I_\infty \to Mf_1$ are paths of length 0.
\end{prop}
\begin{proof}
    We show that for $l \in \mathbb{N} \cup \{ 0 \}$, the following diagram commutes in $\hgrh_*$:
    \begin{equation*}
        \xymatrix{
            \Omega_l G \ar[r]^{\Omega_l f_1} \ar[d]_{j'} & \Omega_l H \ar[r]^{k} \ar[d]_{j} & M_lf_1 \ar[r]^{f_2} \ar[d]_{\id} & G \ar[r]^{f_1} \ar[d]_{\id} & H \ar[d]_{\id} \\
            M_lf_3 \ar[r]_{f_4} & M_lf_2 \ar[r]_{f_3} & M_lf_1 \ar[r]_{f_2} & G \ar[r]_{f_1} & H
        }
    \end{equation*}
    where $\Omega_l G$, $\Omega_l H$, $M_lf_i$ are the induced subgraphs of $\Omega G$, $\Omega H$ and $Mf_i$ consisting of paths of length $l$ respectively.
    First, let us show the commutativity of the leftmost diagram.
    Let $\beta: I_\infty \to G$ be a loop of length $l$.
    Then, $\Omega_l f_1 (\beta) = f_1 \beta$, and thus we have $j \Omega_l f_1 (\beta) = j (f_1 \beta) = (u_0, f_1 \beta, \beta_0)$, where $\beta_0: I_\infty \to G$ is path of length 0.
    On the other hand, $j' (\beta) = (u_0, \omega_0, \beta, \gamma_0)$.
    Thus, $f_4 j' (\beta) = (u_0, \omega_0, \beta)$. \\
    Since $f_1 \beta$ is a loop such that all of its subloops are of sublength $\leq 4$ by assumption, $f_1 \beta$ is null-$A$-homotopic, i.e., there exists a graph map $H: \Omega_l H \otimes I_m \to \Omega_l H$ for some $m$ such that $H(\bullet, 0) = f_1 \beta$ and $H(\bullet, m) = \omega_0$.
    On the other hand, by assumption, we also have another graph map $H': \Omega_l G \otimes I_{m'} \to \Omega_l G$ for some $m'$ such that $H'(\bullet, 0) = \beta_0$ and $H'(\bullet, m') = \beta$.
    Then, consider a map $F: \Omega_l G \otimes I_{m+m'} \to M_lf_2$ defined by
    \begin{equation*}
        F(u_0, \beta, s) =
        \begin{cases}
            (u_0, H(\bullet, s), \beta_0)       & 0 \leq s \leq m;      \\
            (u_0, \omega_0, H'(\bullet, s - m)) & m \leq s \leq m + m'.
        \end{cases}
    \end{equation*}
    This satisfies 
    $F(\beta, 0) = (u_0, f_1 \beta, \beta_0)$ and $F(\beta, m + m') = (u_0, \omega_0, \beta)$.
    By construction, $F$ is a graph map; thus, the leftmost diagram commutes in $\hgrh_*$.

    The commutativity of the second left diagram follows in $\grh_*$ immediately.
    Taking $\omega \in \Omega_l H$, we have $k(\omega) = (u_0, \omega)$.
    On the other hand, $j(\omega) = (u_0, \omega, \beta_0)$.
    Since the map $f_3$ sends $(u, \omega, \beta)$ to $(u, \omega)$, by composing the previous map we have $f_3j = k$.
\end{proof}

In ordinary homotopy theory, for pointed spaces $(X, x_0), (Y, y_0)$ and a pointed map $f: X \to Y$, it can be shown that the sequence
\begin{equation*}
    Mf \xrightarrow{f'} X \xrightarrow{f} Y
\end{equation*}
is exact in $\htop_*$.
As its analogy, it is natural to define the map $q: Mf_1 \to H$ by $q(u, \omega) = \omega(l)$ for a path $\omega: I_\infty \to H$ of length $l$, $l \in \NN \cup \{ 0 \}$.

\begin{prop}\label{prop:omega_Mf}
    Let $f_1: G \to H$ be a pointed graph map.
    Suppose that for any loop $\omega \in \Omega H$, all of its subloops are of sublength $\leq 4$, and that $\Omega f_1$ is surjective.
    If the sequence
    \begin{equation*}
        Mf_1 \xrightarrow{f_2} G \xrightarrow{f_1} H
    \end{equation*}
    is exact in $\hgrh_*$.
    Then, the sequence
    \begin{equation*}
        \Omega G \xrightarrow{\Omega f_1} \Omega H \xrightarrow{k} Mf_1 \xrightarrow{f_2} G \xrightarrow{f_1} H
    \end{equation*}
    is exact in $\hgrh_*$.
\end{prop}
\begin{proof}
    By assumption, it suffices to consider the sequences
    \begin{equation*}
        \Omega H \xrightarrow{k} Mf_1 \xrightarrow{f_2} G
    \end{equation*}
    and
    \begin{equation*}
        \Omega G \xrightarrow{\Omega f_1} \Omega H \xrightarrow{k} Mf_1.
    \end{equation*}
    The first sequence can be shown to be exact in $\grh_*$ as follows. \\
    For $\omega \in \Omega H$, we have $f_2 k (\omega) = f_2 (u_0, \omega) = u_0$, and thus $f_2 k = 0$.
    On the other hand, $(u, \omega) \in \ker f_2$ implies that $f_2(u, \omega) = u = u_0$, and hence, $\omega$ is an element in $\Omega H$, i.e., $\im k \supset \ker f_2$ holds.

    Next, consider the sequence
    \begin{equation*}
        [K, \Omega G] \xrightarrow{\Omega f_{1,*}} [K, \Omega H] \xrightarrow{k_*} [K, Mf_1]
    \end{equation*}
    for any pointed graph $K$.
    Take $\varphi \in [K, \Omega G]$.
    For $w \in K$, we have $\varphi(w) = \beta \in \Omega G$ for some loop $\beta$.
    Then, we have $k \Omega f_1 \beta = k (f_1 \beta) = (u_0, f \beta) \in Mf_1$.
    By assumption, $f \beta$ is a loop such that all its subloops are of sublength $\leq 4$.
    Thus, $k \Omega f_1 \varphi (w): K \to Mf_1$ is null-$A$-homotopic, i.e., $k_* \Omega f_{1,*} (\varphi) = 0$.
    On the other hand, take $\varphi: K \to \Omega H$ such that $k \varphi$ is null-$A$-homotopic.
    By assumption, there exists a loop $\beta \in \Omega G$ such that $\Omega f_1 \beta = \varphi(w)$, thus we have $\im k_* \supset \ker \Omega f_{1,*}$.
\end{proof}

\begin{rem}
    Note that the sequence
    \begin{equation*}
        Mf_1 \xrightarrow{f_2} G \xrightarrow{f_1} H
    \end{equation*}
    is not necessarily exact in $\hgrh_*$.
    To see this, we use the following notation.
    For a pointed graph map $f_1: G \to H$, let $\omega_{f_1(i)}^l$ be a path $I_\infty \to H$ of length $l$ satisfying the condition
    \begin{equation*}
        \omega_{f_1(u)}^l(t) = \begin{cases}
            0,      & t \leq 0, \\
            f_1(u), & t \geq l.
        \end{cases}
    \end{equation*}
    Then, the vertices of the graph $Mf_1$ can be written as $(u, \omega_{f_1(u)}^l)$, $u \in G$, $l \in \mathbb{N} \cup \{ 0 \}$.
    In particular, the connected component containing the vertices $(u, \omega_{f_1(u)}^l)$, $l \in \mathbb{N} \cup \{ 0 \}$ with $f_1(u) = v_0$, $u \in G$ has the base vertex $(0, \omega_0^0)$.
    In $\hgrh_*$, for the sequence
    \begin{equation*}
        [K, Mf_1] \xrightarrow{f_{2,*}} [K, G] \xrightarrow{f_{1,*}} [K, H],
    \end{equation*}
    $\im f_{2,*} \subset \ker f_{1,*}$ holds.
    For $\varphi \in [K, Mf_1]$ and $w \in K$, the image of $\varphi$ is of the form $\varphi(w) = (\eta(w), \omega_{f_1 \eta (w)}^l)$ for some graph map $\eta: K \to G$ and some path $\omega_{f_1 \eta (u)}^l$ of length $l$, $l \in \mathbb{N} \cup \{ 0 \}$.
    Since $\varphi$ is a graph map, the map should satisfy the condition that if $w \sim w'$, then
    \begin{align*}
        \varphi(w) = (\eta(w), \omega_{f_1 \eta(w)}^l) = (\eta(w'), \omega_{f_1 \eta(w')}^{l'}) = \varphi(w') & \text{ or} \\
        \varphi(w) = (\eta(w), \omega_{f_1 \eta(w)}^{l}) \sim (\eta(w'), \omega_{f_1 \eta(w')}^{l'}) = \varphi(w')
    \end{align*}
    holds.
    For the latter, there are two cases: (I) $\eta(w) = \eta(w')$ and $\omega_{f_1 \eta(w)}^l \sim \omega_{f_1 \eta(w')}^{l'}$, and (II) $\eta(w) \sim \eta(w')$ and $\omega_{f_1 \eta(w)}^l = \omega_{f_1 \eta(w')}^{l'}$.
    In particular, for the case (I), since $\omega_{f_1 \eta(w)}^l \sim \omega_{f_1 \eta(w')}^{l'}$ holds, we have $f_1 \eta(w) = f_1 \eta(w')$ otherwise the two paths $\omega_{f_1 \eta(w)}^l, \omega_{f_1 \eta(w')}^{l'}$ are not $A$-homotopy equivalent.
    Moreover, since the graph map is pointed, the vertices of the graph $K$ are sent to the connected component of $Mf_1$ containing the base vertex $(0, \omega_0^0)$.
    Hence, we have $f_1 \eta(w) = f_1 \eta(w') = v_0$, that is, $f_{1}f_{2} \varphi (w) = f_1f_2 (\eta(w), \omega_{f_1 \eta(w)}^l) = f_1 \eta(w) = v_0$, i.e., $\im f_{2,*} \subset \ker f_{1,*}$.

    On the other hand, $\ker f_{1,*} \subset \im f_{2,*}$ does not necessarily hold.
    Take $\eta \in \ker f_{1,*}$.
    Then, there exists some $m \in \mathbb{N}$ and a graph map $F: \Omega_l H \otimes I_m \to \Omega_l H$ such that $F(\bullet, 0) = f_1 \eta$ and $F(\bullet, m) = \omega_0$.
    To obtain the desired property, we need to define a graph map $\varphi: K \to Mf_1$ by $\varphi(w) = (\eta(w), \omega_{f_1\eta(w)}^l)$ for some $l$, however, the existence of a path $\omega_{f_1\eta(w)}^l$ giving $\varphi$ a graph map structure is not guaranteed.
\end{rem}

Let us define reduced suspension graphs (see also \cite{Lb} for the definition and observations of the unreduced suspension graph).
\begin{defn}
    For a pointed graph $G$, define a \textbf{reduced suspension graph} $\Sigma G$ as the graph obtained from $G \otimes I_\infty$ by contracting all the vertices $(u, t)$ , $u \in G$, $t \leq 0$, $t \geq l$, and $(u_0, t)$ to a vertex $(u_0, 0)$.
    For fixed $l \in \mathbb{N} \cup \{ 0 \}$, we call the graph a reduced suspension graph of length $l$ and denote the graph by $\Sigma_l G$
\end{defn}
\begin{prop}\label{prop:adjoint}
    Adjoint relation $\left[ \Sigma G, H  \right] = \left[ G, \Omega H  \right]$ holds.
\end{prop}
\begin{proof}
    Construct $\Phi: \left[ \Sigma_l G, H  \right] \to \left[ G, \Omega_l H  \right]$ by for $f \in \left[ \Sigma_l G, H  \right]$, $\Phi(f)(u)(t) = f(u, t)$.
    The map is well-defined and a bijection.
    Actually, we can construct its inverse $\Psi: \left[ G, \Omega_l H  \right] \to \left[ \Sigma_l G, H  \right]$, by defining for $g: G \to \Omega_l H$, $\Psi(g)(u, t) = g(u)(t)$.
    Then, we have
    \begin{align*}
        \left( \Psi \Phi ( f ) \right)(u, t) = \Psi \left( \Phi( f ) \right) (u, t) = \Phi (f)(u) (t) = f(u, t) \\
        \left( \left( \Phi \Psi \right) (g) \right) (u)(t) = \Phi \left( \Psi (g) \right) (u)(t)  = \Psi(g) (u, t) = g(u)(t)
    \end{align*}
    Thus, $\left( \Psi \Phi \right)(f) = f$, $\left( \Psi \Phi \right)(g) = g$, and so $\Phi$ is a bijection.
\end{proof}
\begin{prop}\label{prop:looped}
    If $G' \to G \to G''$ is exact in $\hgrh_*$, then the ``looped'' sequence
    \begin{equation*}
        \Omega G' \to \Omega G \to \Omega G''
    \end{equation*}
    is also exact in $\hgrh_*$.
\end{prop}
\begin{proof}
    By Proposition \ref{prop:adjoint}, we have a commutative diagram whose rows are pointed bijections for any pointed graph $K$:
    \begin{equation*}
        \xymatrix{
        [\Sigma K, G'] \ar[r] \ar[d] & [\Sigma K, G] \ar[r] \ar[d] & [\Sigma K, G''] \ar[d] \\
        [K, \Omega G'] \ar[r] & [K, \Omega G] \ar[r] & [K, \Omega G''].
        }
    \end{equation*}
    By the assumption, the first row is exact in $\sets_*$, and thus the second row is also exact in $\sets_*$.
\end{proof}

Now, we give the main theorem of the present paper.
\begin{thm}[Discrete Puppe sequence]
    Let $f: G \to H$ be a pointed graph map.
    Suppose that loops $\omega \in \Omega^n H$, $\beta \in \Omega^n G$ have only subloops of sublength $\leq 4$, and $\Omega^n f_1$ is surjective for all $n \in \mathbb{N}$.
    If the sequence
    \begin{equation*}
        Mf_1 \xrightarrow{f_2} G \xrightarrow{f_1} H
    \end{equation*}
    is exact in $\hgrh_*$, then the following sequence is exact in $\hgrh_*$:
    \begin{align*}
        \cdots \xrightarrow{\Omega^2 k} \Omega^2 (Mf_1) \xrightarrow{\Omega^2 f_2} \Omega^2 G \xrightarrow{\Omega^2 f_1} \Omega^2 H \xrightarrow{\Omega k} \Omega(Mf_1) \xrightarrow{\Omega f_2}                               \Omega G
        \xrightarrow{\Omega f_1}                                                                                         \Omega H \xrightarrow{k} Mf_1 \xrightarrow{f_2} G                                  \xrightarrow{f_1} H.
    \end{align*}
\end{thm}
\begin{proof}
    The sequence
    \begin{equation*}
        \Omega G \xrightarrow{f_1} \Omega H \xrightarrow{k} Mf_1 \xrightarrow{f_2} G \xrightarrow{f_1} H
    \end{equation*}
    is exact in $\hgrh_*$ by Proposition \ref{prop:omega_Mf}.
    By Proposition \ref{prop:looped}, the sequence
    \begin{equation*}
        \Omega^2 G \xrightarrow{\Omega^2 f_1} \Omega^2 H \xrightarrow{\Omega k} \Omega (Mf_1) \xrightarrow{\Omega f_2} \Omega G
    \end{equation*}
    is also exact in $\hgrh_*$.
    Finally, by connecting these sequences and by induction, the theorem follows.
\end{proof}
\begin{rem}
    In \cite{LWYZ}, homotopy groups of digraphs, a variation of the GLMY homotopy groups has been introduced by Grigor'yan, Lin, Muranov, and Yau.
    They also proved the Puppe sequence holds under the homotopy of digraphs.
    The morphism they use is weak homotopy, i.e., the digraph map which induces the isomorphism on homotopy groups of all degrees.
    While the present paper is based on strict homotopy, and the analogous statement holds but technical assumptions are needed.
    Further study is needed to investigate the relation between their work and the present result.
\end{rem}
\bibliographystyle{amsalpha}
\bibliography{discrete_htp.bib}
\end{document}